\documentclass{ip-journal}
\usepackage{amscd}

\def\sO{\mathcal O}
\def\sV{\mathcal V}
\def\sH{\mathcal H}
\def\sI{\mathcal I}
\def\sK{\mathcal K}
\def\sL{\mathcal L}
\def\sQ{\mathcal Q}
\def\sE{\mathcal E}
\def\sF{\mathcal F}
\def\bP{\mathbb P}
\def\bA{\mathbb A}
\def\bC{\mathbb C}
\def\bZ{\mathbb Z}

\def\logcot{\Omega_X(\log D)}
\def\map{\longrightarrow}

\def\Hom{\mathrm{Hom}}
\def\Ext{\mathrm{Ext}}

\title[a construction of stable bundle]{A construction of stable bundles and reflexive sheaves on Calabi-Yau threefolds}
\author{Baosen Wu and Shing Tung Yau}

\address{Department of Mathematics, Harvard University, Cambridge, MA 02138}
\email{bwu@math.harvard.edu}

\address{Department of Mathematics, Harvard University, Cambridge, MA 02138}
\email{yau@math.harvard.edu}

\newtheorem{prop}{Proposition}
\newtheorem{coro}{Corollary}
\newtheorem{thm}{Theorem}
\newtheorem{conj}{Conjecture}
\newtheorem*{thm*}{Theorem}

\begin{document}

\begin{abstract}
We use Serre construction and deformation to construct stable bundles and reflexive sheaves on Calabi-Yau threefolds.
\end{abstract}

\maketitle

\section*{Introduction}

The goal of this note is to use a deformation technique to construct slope stable bundles and reflexive sheaves on an arbitrary (projective) Calabi-Yau threefold with prescribed Chern classes within a certain range.

For a projective algebraic surface $X$, Gieseker \cite{Gie88} showed that if $c_2>2p_g(X)+4$, then there exists a rank $2$ stable bundle $E$ over $X$ with $c_1(E)=0$ and $c_2(E)=c_2$ using deformation theory of sheaves, which was motivated from the differential geometric counterpart obtain by Taubes \cite{Tau84}. A similar result in the higher rank case was proved by Artamkin \cite{Art90}.

Now for threefolds, we don't have corresponding existence result in such a general formulation. It is well-known that Serre construction and Monad construction are successful tools to construct vector bundles, especially on $\bP^3$. However, for a general Calabi-Yau threefold, little is known about the construction of stable bundles with prescribed Chern classes. In 2006, Douglas, Reinbacher and Yau \cite{DRY06} proposed the following conjecture (abbreviated as DRY-conjecture below) concerning the existence of stable reflexive sheaves on Calabi-Yau threefolds. From now on, we restrict ourself to the case of $c_1=0$.

\begin{conj}\cite[Conjecture 1.1]{DRY06}
Let $X$ be a simply connected Calabi-Yau threefold. Let $c_i\in H^{2i}(X, \bZ)$, i=2,3, and $r>1$ be an integer. Suppose there is an ample class $H\in H^2(X,\mathbb Q)$ such that
\[ c_2=r\big(H^2+\frac{1}{24}c_2(X)\big),\quad c_3<\frac{2^{9/2}}{3}rH^3,\]
then there exists a rank $r$ stable reflexive sheaf $\sE$ on $X$ w.r.t. some polarization and $c_1(\sE)=0$, $c_2(\sE)=c_2$ and $c_3(\sE)=c_3$.
\end{conj}

The main result in this note is Theorem \ref{main} in section $2$. It gives a sufficient condition on a curve $C$ in a Calabi-Yau threefold $X$ so that $\sO_X\oplus \sI_C$ deforms to a stable reflexive sheaf. As an application, we get the following theorem.

\begin{thm*}[Theorem \ref{result}]
Let $D\subset X$ be a very ample divisor on a Calabi-Yau threefold $X$. Then there exists a rank $r$ stable reflexive sheaf $\sE$ on $X$ with $c_1(\sE)=0$, $c_2(\sE)=n[D^2]$ and $c_3(\sE)=2nD^3$ for $n\ge r$.
\end{thm*}

Note that we use $[D^2]$ to denote a curve class and its Poincar\'{e} dual cohomology class interchangeably.

In general, testing stability of a bundle or a sheaf is a difficult problem. Our approach is inspired by the idea of Huybrechts \cite{Huy95}, where he studied sufficient conditions that $\sO_X\oplus T_X$ deforms to a stable bundle on a Calabi-Yau threefold $X$. The advantage of this approach, just as the one used by Gieseker and Artamkin, is that the constructed reflexive sheaf is stable with respect to \emph{all} polarization of $X$.

To deduce Theorem $2$, we first use the fact that $D$ is very ample to choose a suitable representative curve in the class $n[D^2]\in H_2(X,\bZ)$ with $n\ge 2$ so that it fulfill the conditions in Theorem $1$. Then we get a rank $2$ stable reflexive sheaf with $c_1=0$ and $c_2=n[D^2]$. Finally we complete the proof by induction on the rank.

This paper is organized as follows. In section $1$ we study the deformation of $\sO_X\oplus T_X$ to a stable bundle on a Calabi-Yau threefold. This question has other motivation and applications. Here we use it as a toy model for the later construction. In section $2$ we generalize the method to construct stable reflexive sheaves. Although the approach is motivated from the result in section $1$, this section is self-contained. Finally in section $3$ we discuss briefly further research directions. Some applications to elliptic fibered Calabi-Yau will come up later.

Acknowledgement. B.W. would like to thank Peng Gao, Xiaowei Wang for helpful discussions. This work was done when the authors were visiting the Department of mathematics in National Taiwan University, we would like to thank for the hospitality of the department during their visit.

\section{Deforming $\sO_X\oplus T_X$ to a stable bundle}

In 1980's, Witten \cite{Wit86} asked the following question: can one deform the bundle $\sO_X\oplus T_X$ on a Calabi-Yau threefold to a stable bundle so that its restriction to any rational curve is nontrivial? This is list as Problem 77 in Yau's problem list \cite{Yau93}.

Due to the lack of knowledge about rational curves in Calabi-Yau threefolds, this problem is still open. However, Huybrechts \cite{Huy95} made some progress for certain Calabi-Yau threefolds. Now if we consider only the first half of the question: can one deform the direct sum $\sO_X\oplus T_X$ to a slope stable bundle? then \cite{Huy95} and \cite{LiYau05} gave satisfactory answers already.

We aim to generalize Huybrechts' result in the following formulation. This also serves as a warm-up for the construction in the next section. For convenience, we shall work on the dual bundle $\sO_X\oplus \Omega_X$ instead of $\sO_X\oplus T_X$.

\begin{prop}\label{log}
Let $X$ be a Calabi-Yau threefold with a smooth ample divisor $D$. Suppose $H^0(\Omega_X(\log D))=0$, and the natural map $H^1(T_X(-D))\map H^1(T_X)$ is nonzero, then there exists an extension
\[ 0\map \sO_X(-D) \map E \map \logcot \map 0 \]
so that $E$ is slope stable with respect to any polarization of $X$. Furthermore, $E$ is a deformation of $\sO_X\oplus\Omega_X$.
\end{prop}

\begin{proof}
The strategy of the proof is as follows. We construct a family of vector bundles $\sE$ over $X$ parametrized by $\bP^{\! 1}$, so that $\sE_0$ fits into the following non-splitting short exact sequences
\begin{equation}
\begin{CD}\label{theta}
\theta_1:\quad 0\map \sO_X\map \sE_0\map \Omega_X\map 0 \\
%\theta_2:\quad 0\map \Omega_X\map \sF_0\map \sO_X\map 0
\end{CD}
\end{equation}
Moreover, we demand that for each $t\ne 0$, $\sE_t$ fits into the following
\[ 0\map \sO_X(-D) \map \sE_t \map \logcot \map 0. \]
It follows from the assumption that $\sE_t$ has no global section. We take a generic $\sE_t$ as the required $E$. To check its stability, we shall use the fact that semistability is an open property.

\vskip 5pt

Now we construct the family $\sE$. By assumption, the map $H^1(T_X(-D))\map H^1(T_X)$, or equivalently, $\Ext^1(\Omega_X,\sO_X(-D))\map \Ext^1(\Omega_X,\sO_X)$ is nonzero, we can find $\rho\in\Ext^1(\Omega_X,\sO_X(-D))$ so that its image $\theta_1\in \Ext^1(\Omega_X,\sO_X)$ is not zero. Therefore we get vector bundles $G$ and $\sE_0$, fitting into the following commutative diagram
\begin{equation}
\begin{CD}\label{comm1}
\rho:\quad 0@>>>\sO_X(-D)@>>>G@>>>\Omega_X@>>>0\\
@. @VVV @VVV @| \\
\theta_1:\quad 0@>>>\sO_X@>>>\sE_0@>>>\Omega_X@>>>0\\
\end{CD}
\end{equation}
such that the horizontal sequences are non-splitting. Here the bottom sequence defines the required element $\theta_1$ in (\ref{theta}). Taking cokernel of the middle vertical map, we get
\begin{equation}\label{eta0}
\eta_0:\quad 0\map G\map \sE_0\map \sO_D\map 0.
\end{equation}

To define $\sE$, we construct another element $\eta_1\in \Ext^1(\sO_D, G)$
\begin{equation}\label{eta1}
\eta_1:\quad 0\map G\map \sE_1\map \sO_D\map 0,
\end{equation}
and let $\sE_t$ be the sheaf defined by the extension $(1-t)\eta_0+t\eta_1\in \Ext^1(\sO_D, G)$. For this, we apply $\Hom(-,\sO_X(-D))$ to the standard sequence
\begin{equation}\label{gamma}
\gamma:\quad 0\map\Omega_X\map\Omega_X(\log D)\map \sO_D\map 0,
\end{equation}
then we obtain the following maps
\begin{equation}\label{map}
\Ext^1(\sO_D,\sO_X(-D))\stackrel{f_1}\map \Ext^1(\Omega_X(\log D),\sO_X(-D))\stackrel{f_2}\map \Ext^1(\Omega_X,\sO_X(-D)).
\end{equation}

Note that $f_2$ is surjective, since $\Ext^2(\sO_D,\sO_X(-D))=0$ by the ampleness of $D$ and Kodaira vanishing theorem. For $\rho$ in (\ref{comm1}), we let $\xi_1\in f_2^{-1}(\rho)$. Then we obtain the following diagram
\begin{equation}
\begin{CD}\label{comm2}
\rho:\quad 0@>>>\sO_X(-D)@>>>G@>>>\Omega_X@>>>0\\
@. @| @VVV @VVV \\
\xi_1:\quad 0@>>>\sO_X(-D)@>>>\sE_1@>>>\Omega_X(\log D)@>>>0\\
\end{CD}
\end{equation}
Taking cokernel of the middle vertical map, we get the required element $\eta_1$ in (\ref{eta1}).

Since $H^0(\sE_0)=\bC$ by the sequence $\theta_1$ in (\ref{comm1}), and $H^0(\sE_1)=0$ by the sequence $\xi_1$ in (\ref{comm2}) and the assumption, and both $\eta_i$ are nonzero, we know that $\eta_0$ and $\eta_1$ are linearly independent in $\Ext^1(\sO_D, G)$, and therefore, they generate a 2-dimensional subspace $V_1$. The required family $\sE$ is defined as the family of vector bundles over $X$ parametrized by $\bP^{\! 1}=\bP(V_1)$ whose fiber $\sE_t$ is given by the sheaf (which is locally free) defined by the extension $(1-t)\eta_0+t\eta_1\in V_1$ .

Finally, we show that $\sE_t$ is stable for $t\ne 0$. Since $\sE_0$ is (strictly) semistable, and by \cite{Mar76} semistability is an open property, we know that $\sE_t$ is semistable. If it is strictly semistable, then by $H^1(\sO_X)=0$ and the stability of $\Omega_X$, we either have a nonzero homomorphism $\sO_X\to \sE_t$, or a nonzero homomorphism $\sE_t\to \sO_X$. Note that by the nonsplitting of the sequence $\theta_1$ in (\ref{comm1}), we have $H^0((\sE_0)^\vee)=0$. It follows from upper semi-continuity that $H^0((\sE_t)^\vee)=0$ for generic $t\ne 0$. Therefore, $\sE_t$ has no global section and nonzero homomorphism to $\sO_X$. Hence it is stable. Simply take $E$ be such an $\sE_t$, we get the required stable bundle.
\end{proof}

In the remainder of this section, we discuss the relation of Proposition \ref{log} with the results of Huybrechts and give an example on quintic threefold. This part will not be used later in this note, the readers who are interested in the construction of stable bundles and reflexive sheaves can skip to the next section.

We first make some remarks about this proof.

(1) The proof works for singular ample divisor $D$ as well. Indeed, what we need in the proof is to find a nontrivial extension
\[ 0\map\Omega_X\map \sV\map \sO_D\map 0, \]
so that $H^0(\sV)=0$. The role played by $\Omega_X(\log D)$ is replaced by that of $\sV$.

(2) The ampleness of $D$ is not essential in the proof. All we need is
\[\Ext^2(\sO_D,\sO_X(-D))=0, \]
which is equivalent to $H^1(\sO_D)=0$ since $X$ is a Calabi-Yau threefold. This implies that the proof works even for some non-projective (nonk\"{a}hler) case as well.

\vskip 5pt

To compare with Huybrechts' results, recall that he \cite{Huy95} gave the following sufficient conditions so that $\sO_X\oplus \Omega_X$ deforms to a stable bundle, in the case where $X\subset M$ is a Calabi-Yau manifold embedded as a hypersurface in $M$:

(I) $0\ne c_1(N_{X/M})\in H^1(\Omega_X)$, where $N_{X/M}=\sO_M(X)|_X$ is the normal bundle of $X$.

(II) The natural map $H^0(N_{X/M})\map H^1(T_X)$ is nonzero.

One can show that if a Calabi-Yau threefold satisfies Huybrechts' conditions (I) and (II), then it also satisfies the assumptions in Proposition \ref{log}, where we take $D$ as the zero divisor of a section of $N_{X/M}$ whose existence is guaranteed by condition (II). We also provide an example (Corollary \ref{example} below) which is not covered by his method. In this sense, Proposition \ref{log} can be regarded as a generalization of his approach.

Before we proceed, we recall Huybrechts' idea. He constructed two families of vector bundles $\sE$ and $\sF$ over $X$ parametrized by $\bA^1$, so that $\sE_t\cong \sF_t$ for $t\ne 0$, and $\sE_0$ is a nontrivial extension of $\Omega_X$ by $\sO_X$, and $\sF_0$ is a nontrivial extension of $\sO_X$ by $\Omega_X$. By the openness of semi-stability, he proved that $\sE_t$ is stable for generic $t$.

Our approach is essential the same. We have constructed the family $\sE$ in the proof. Now we construct another family $\sF$ so that $\sF_0$ fits into a nontrivial
\begin{equation}\label{theta2}
\theta_2:\quad 0\map \Omega_X\map\sF_0\map\sO_X\map 0.
\end{equation}
Recall that the map $f_1$ in (\ref{map}) is injective, the standard sequence
\[0\map \sO_X(-D)\map \sO_X\map \sO_D\map 0 \]
defines a nonzero element in $\Ext^1(\sO_D,\sO_X(-D))$, its image under $f_1$, say $\xi_0$, is a nontrivial element in $\Ext^1(\Omega_X(\log D),\sO_X(-D))$. In conclusion, we have the following commutative diagram
\begin{equation}
\begin{CD}
@. @.\Omega_X @= \Omega_X\\
@.@.@VVV @VVV\\
\xi_0:\quad 0@>>>\sO_X(-D)@>>>\sF_0@>>>\Omega_X(\log D)@>>>0\\
@. @| @VVV @VVV \\
\quad\quad\quad 0@>>>\sO_X(-D)@>>>\sO_X@>>>\sO_D@>>>0\\
\end{CD}
\end{equation}
Here the middle vertical sequence is the required element $\theta_2$ in (\ref{theta2}).

Note that we have obtained an element $\xi_1$ in (\ref{comm2}). Similarly, we can show that $\xi_0$ and $\xi_1$ generate a 2-dimensional subspace $V_2\subset \Ext^1(\Omega_X(\log D),\sO_X(-D))$. We define $\sF$ be the family of vector bundles on $X$ parametrized by $\bP^{\! 1}=\bP(V_2)$ so that the fiber $\sF_t$ is the bundle obtain by the extension $(1-t)\xi_0+t\xi_1\in V_2$.

\vskip 5pt

It remains to show that $\sE_t\cong \sF_t$ for $t\ne 0$. Note that for $\xi_t=(1-t)\xi_0+t\xi_1$, we have $f_2(t^{-1}\xi_t)=f_2(\xi_1)=\rho$. Therefore we get
\begin{equation}
\begin{CD}\label{comm3}
\rho:\quad\quad\quad 0@>>>\sO_X(-D)@>>>G@>>>\Omega_X@>>>0\\
@. @| @VVV @VVV \\
t^{-1}\xi_t:\quad 0@>>>\sO_X(-D)@>>>\sF_t@>u>>\Omega_X(\log D)@>>>0\\
@.@.@VVV @VVvV\\
@.@.\sO_D@=\sO_D
\end{CD}
\end{equation}
Applying $\Hom(\sO_D,-)$ to the short exact sequence given by $\rho$ in the top horizontal sequence above, we get
\begin{equation}
0\map\Ext^1(\sO_D,\sO_X(-D))\map \Ext^1(\sO_D,G)\stackrel{g}\map \Ext^1(\sO_D,\Omega_X)
\end{equation}
For $\gamma$ in (\ref{gamma}), let $W_1\subset \Ext^1(\sO_D,G)$ be the subset of elements $\xi$ with $g(\xi)=\gamma$.
Let $W_2\subset \Ext^1(\Omega_X(\log D),\sO_X(-D))$ be the subset of elements $\xi$ with $f_2(\xi)=\rho$. Notice that $W_i\subset V_i$ are 1-dimensional affine subspaces, and after projectivization, we have $W_i\cong\bP(V_i)-0$.

From diagram (\ref{comm3}), by composing the homomorphisms $u$ and $v$, we see that there is a one-to-one affine map $W_2\map W_1$ which send $\xi_1$ to $\eta_1$. It further implies that $t^{-1}\xi_t$ maps to $s^{-1}\eta_s$ for some $s\ne 0$. Therefore, possibly after relabeling the index, we have $\sE_t\cong\sF_t$ for $t\ne 0$.

\vskip 5pt

To conclude this section, we provide an example.

\begin{coro}\label{example}
Let $X\subset \bP^4$ be a quintic threefold. Let $H\subset \bP^4$ be a hyperplane. Denote $D=X\cap H$. Then we have $H^0(\Omega_X(\log D))=0$, and the natural map $H^1(T_X(-D))\to H^1(T_X)$ is nonzero. Therefore $\sO_X\oplus \Omega_X$ deforms to a stable bundle $E$ fitting into the following
\[ 0\map \sO_X(-D) \map E \map \logcot \map 0. \]
\end{coro}

\begin{proof}
First we show that $H^0(\Omega_X(\log D))=0$. Consider the natural map
\[\Ext^1(\sO_{\bP^4},\Omega_{\bP^4})\map \Ext^1(\sO_{\bP^4},\Omega_{\bP^4}(\log H)).\]
The Euler sequence on $\bP^4$ defines an element in $\Ext^1(\sO_{\bP^4},\Omega_{\bP^4})$. Then we obtain the following diagram
\begin{equation}
\begin{CD}
0@>>>\Omega_{\bP^4}@>>>\sO_{\bP^4}(-1)^{\oplus 5}@>>>\sO_{\bP^4}@>>>0 \\
@.@VVV @VVV @| \\
0@>>>\Omega_{\bP^4}(\log H)@>>> E @>>>\sO_{\bP^4}@>>>0 \\
@.@VVV @VVV \\
@.\sO_H@=\sO_H
\end{CD}
\end{equation}
It follows that $E\cong \sO_{\bP^4}\oplus \sO_{\bP^4}(-1)^{\oplus 4}$, and $\Omega_{\bP^4}(\log H)\cong \sO_{\bP^4}(-1)^{\oplus 4}$. Now from the following
\[ 0\map \sO_X(-5)\map \Omega_{\bP^4}(\log H)|_X\map \Omega_{X}(\log D)\map 0, \]
we obtain $H^0(\Omega_{X}(\log D))=H^0(\Omega_{\bP^4}(\log H)|_X)=0$.

Next we show that $H^1(T_X(-D))\to H^1(T_X)$ is not zero. It is equivalent to the nonvanishing of $H^2(\Omega_X)\to H^2(\Omega_X(D))$. We can find a bundle $F$ which fits into the following diagram
\begin{equation}
\begin{CD}
@.\sO_X(-5)@=\sO_X(-5)@. \\
@.@VVV @VVV @. \\
0@>>>\Omega_{\bP^4}|_X@>>>\sO_{\bP^4}(-1)^{\oplus 5}|_X  @>>>\sO_{\bP^4}|_X@>>>0 \\
@.@VVV @VVV @|\\
0@>>>\Omega_X@>>>F@>>>\sO_X@>>>0
\end{CD}
\end{equation}
Then by $H^i(\sO_X)=H^i(\sO_X(D))=0$ for $i=1,2$, it is equivalent to show that $H^2(F)\to H^2(F(D))$ is nonzero. Now we obtain
\begin{equation}
\begin{CD}
0@>>>\sO_X(-5)@>>>\sO_{X}(-1)^{\oplus 5}@>>>F@>>> 0 \\
@.@VVV @VVV @VVV \\
0@>>>\sO_X(-4)@>>>\sO_{X}^{\oplus 5}@>>>F(D)@>>> 0 \\
\end{CD}
\end{equation}
Taking cohomology of the bundles, we get
\begin{equation}\notag
\begin{CD}
0@>>>H^2(F)@>>>H^3(\sO_X(-5))@>>>H^3(\sO_{X}(-1)^{\oplus 5}) \\
@.@VVV @VVV @VVV \\
0@>>>H^2(F(D))@>>>H^3(\sO_X(-4))@>>>H^3(\sO_{X}^{\oplus 5})  \\
\end{CD}
\end{equation}
A dimension calculation shows that the left vertical map is nonzero. Hence we can apply Proposition \ref{log} to get the result.
\end{proof}

\section{Stable reflexive sheaves on Calabi-Yau threefolds}

In this section we shall construct stable reflexive sheaves on Calabi-Yau threefolds with some given Chern classes.

Recall that a sheaf $\sE$ on $X$ is reflexive if $\sE_p$ is a reflexive $\sO_p$-module for every $p\in X$; and an $A$-module $M$ is reflexive if the canonical homomorphism $M\to (M^{\!\vee})^{\!\vee}$ is an isomorphism. It is well-known that for a reflexive sheaf $\sE$ on a smooth variety $X$, the singular locus $\{p\in X \,\,|\,\, \sE_p \text{ is not locally free}\}$ has codimension at least $3$ in $X$; and every rank $1$ reflexive sheaf is invertible.

%\begin{lemm}\label{ref}
%Let $\sF$ be a reflexive sheaf on a Calabi-Yau threefold $X$. Then

%(1) there exists a two term locally free resolution
%\[ 0\map \sV_1\map\sV_0\map \sF\map 0;\]

%(2) the extension $0\to \sF\to \sE\to\sF'\to 0$ is also reflexive if $\sF'$ is reflexive or $\sF'=\sO_D$ for a divisor $D\subset X$.
%\end{lemm}

Serre construction is a useful tool to construct rank $2$ vector bundles on threefolds. Hartshorne \cite{Hat80} generalized this construction to reflexive sheaves on $\bP^3$, and Vermeire \cite{Ver07} to an arbitrary projective threefold. In particular, for Calabi-Yau threefolds, we have the following

\begin{prop}\cite[Theorem 1]{Ver07}\label{Hart}
Let $X$ be a projective Calabi-Yau threefold. Let $\sL$ be an invertible sheaf with $H^1(X,\sL^{-1})=H^2(X,\sL^{-1})=0$. Then there is a one-to-one correspondence between

(1) a rank $2$ reflexive sheaf $\sF$ on $X$ with $c_1(\sF)=\sL$, and a section $s\in H^0(\sF)$ whose zero locus has codimension $2$, and

(2) a Cohen-Macaulay curve $C\subset X$, generically locally complete intersection, and a global section $\xi\in H^0(\omega_C\otimes \sL_C^{-1})$ which generates $\omega_C\otimes \sL_C^{-1}$ except at finitely many points, where $\sL_C=\sL|_C$.
\end{prop}

Recall that a curve is Cohen-Macaulay if it has equidimension one and has no embedded point.

This correspondence is straightforward. Given a sheaf $\sF$ with a section $s\in H^0(\sF)$, we have
\begin{equation}\label{ext1}
0\map\sO_X\stackrel{s}\map\sF\map\sI_C\otimes \sL\map 0.
\end{equation}
The cokernel of $s$ has the form $\sI_C\otimes \sL$ because the zero locus of $s$ has codimension $2$. It defines an element in
\[\Ext^1(\sI_C\otimes\sL,\sO_X)\cong H^2(\sI_C\otimes\sL)^{\!\vee}\cong H^1(\sL_C)^{\!\vee}\cong H^0(\omega_C\otimes\sL_C^{-1}).\]
Conversely, given a curve $C\subset X$ and a section $\xi\in H^0(\omega_C\otimes\sL_C^{-1})$ as stated in (2) of Proposition \ref{Hart}, it defines an element in $\Ext^1(\sI_C\otimes\sL,\sO_X)$ via the above isomorphism, and hence an extension (\ref{ext1}).

\vskip 5pt

Using this correspondence, we generalize Proposition \ref{log} to obtain the following result.

%\begin{prop}
%Let $X$ be a projective Calabi-Yau threefold. Let $\sF$ be a stable sheaf on $X$ with $c_1(\sF)=0$. Suppose there is an ample divisor $D\subset X$ so that

%(1) there exists an extension $0\map\sF\map \sF'\map \sO_D\map 0$ so that $H^0(\sF')=0$,

%(2) $\Ext^1(\sF,\sO_X(-D))\map \Ext^1(\sF,\sO_X)$ is nonzero.

%Then the direct sum $\sO_X\oplus \sF$ deforms to a stable sheaf. In addition, it is reflexive if $\sF$ is reflexive.
%\end{prop}

\begin{thm}\label{main}
Let $C\subset X$ be a Cohen-Macaulay curve in a Calabi-Yau threefold $X$, generically locally complete intersection. Suppose $D\subset X$ is an ample divisor so that $H^0(\sI_C(D))=0$, and the scheme-theoretic intersection $C\cap D$ is zero dimensional which lies in the zero locus of a section of $N_{D/X}$; suppose further that there is a section in $H^0(\omega_C(-D))$ which generates $\omega_C(-D)$ except at finitely many points. Then $\sO_X\oplus \sI_C$ deforms to a stable reflexive sheaf.
\end{thm}

\begin{proof}
We follow the same idea in the proof of Proposition \ref{log}. First we show that there is a nontrivial extension
\begin{equation}\label{a1}
0\map \sI_C\map \sF\map \sO_D\map 0
\end{equation}
with $H^0(\sF)=0$. Note that $\sF$ plays the role of $\Omega_X(\log D)$ in Proposition \ref{log}. Consider the standard exact sequence
\[0\map \sO_X\map \sO_X(D)\map\sO_D(D)\map 0. \]
By tensoring it with $\sI_C$ and the assumption that $C\cap D$ is zero dimensional, we obtain a short exact sequence
\begin{equation}\label{b1}
0\map \sI_C\map\sI_C(D)\map \sI_{Z\subset D}(D)\map 0
\end{equation}
where $Z=C\cap D\subset D$ is a closed subscheme of $D$ with ideal sheaf $\sI_{Z\subset D}$ on $D$. By assumption, $Z$ is contained in the zero locus of a section of $N_{D/X}$, that is, we have an injective homomorphism $\sO_D\to \sI_{Z\subset D}(D)$. Now we pull back the sequence (\ref{b1}) via this homomorphism, and get
\begin{equation}
\begin{CD}\label{comm11}
0@>>>\sI_C@>>>\sF@>>>\sO_D@>>>0\\
@. @| @VVV @VVV \\
0@>>>\sI_C@>>>\sI_C(D)@>>>\sI_{Z\subset D}(D)@>>>0\\
\end{CD}
\end{equation}
where the top row is the required sequence (\ref{a1}). In addition, $H^0(\sF)=0$ because $\sF$ is a subsheaf of $\sI_C(D)$ and the assumption $H^0(\sI_C(D))=0$.

Next, since the intersection $C\cap D$ is zero dimensional, the natural homomorphism $H^0(\omega_C(-D))\to H^0(\omega_C)$ is injective. Using Serre duality, we have isomorphisms
\begin{equation}
\begin{CD}\label{comm111}
H^0(\omega_C(-D))@>f>> H^0(\omega_C)\\
@VV\cong V @VV\cong V \\
\Ext^1(\sI_C,\sO_X(-D))@>g>> \Ext^1(\sI_C,\sO_X)
\end{CD}
\end{equation}
By assumption, we have a section $\theta\in H^0(\omega_C(-D))$ which generates $\omega_C(-D)$ except at finitely many points.

\vskip 5pt

Now we apply the strategy in the proof of Proposition \ref{log}. The homomorphism $g$ in diagram (\ref{comm111}) gives a commutative diagram
\begin{equation}
\begin{CD}\label{comm10}
0@>>>\sO_X(-D)@>>>G@>>>\sI_C@>>>0\\
@. @VVV @VVV @| \\
0@>>>\sO_X@>>>\sE_0@>>>\sI_C@>>>0\\
\end{CD}
\end{equation}
where the top row is given by $\theta\in H^0(\omega_C(-D))$ and the bottom row by $f(\theta)$. Since $f(\theta)\in H^0(\omega_C)$ generates $\omega_C$ except at finitely many points, by Proposition \ref{Hart}, we know that $\sE_0$ is reflexive. On the other hand, since $D$ is ample, Kodaira vanishing theorem implies that $\Ext^2(\sO_D,\sO_X(-D))=0$. It follows that the natural homomorphism $\Ext^1(\sF,\sO_X(-D))\to \Ext^1(\sI_C,\sO_X(-D))$ is surjective. Therefore we can find $\sE_1$ fitting into the following
\begin{equation}
\begin{CD}\label{comm20}
\theta:\quad 0@>>>\sO_X(-D)@>>>G@>>>\sI_C@>>>0\\
@. @| @VVV @VVV \\
\quad\quad 0@>>>\sO_X(-D)@>>>\sE_1@>>>\sF@>>>0\\
\end{CD}
\end{equation}

Now consider the central columns of diagram (\ref{comm10}) and (\ref{comm20}), we obtain two extensions
\[\xi_i:\quad 0\map G\map \sE_i\map \sO_D\map 0, \quad i=0,1. \]
We define a family of sheaves $\sE$ on $X$ parametrized by $\bA^1$ so that its fiber $\sE_t$ is a sheaf given by the extension $(1-t)\xi_0+t\xi_1\in\Ext^1(\sO_D,G)$. In addition, by (\ref{comm20}), for $t\ne 0$, $\sE_t$ fits into an exact sequence
\[ 0\map\sO_X(-D)\map\sE_t\map\sF\map 0.\]

Finally, we show that $\sE_t$ is stable for generic $t\in\bA^1$.
From \cite[Theorem 2.8]{Mar76}, semistability is an open property. It follows that if $\sE_t$ is not stable, it must be strictly semistable and we have
\begin{equation}\label{des}
0\map \sL_1\map \sE_t\map\sL_2\map 0,
\end{equation}
for some rank $1$ torsion free sheaves $\sL_i$ on $X$ with $c_1(\sL_i)=0$. Since $\sE_0$ is reflexive, and being reflexive is also an open property \cite[Proposition 1.1.10]{HL10}, $\sE_t$ is reflexive for all but finitely many $t\in \bA^1$. So we can take double dual of the sequence (\ref{des}) and hence assume without loss of generality that $\sL_1$ is invertible. By specializing $t$ to $0$, we know that $\sL_1$ is a deformation of $\sO_X$. Since $H^1(\sO_X)=0$, $\sO_X$ has no nontrivial deformation. It forces $\sL_1=\sO_X$.  On the other hand, since $H^0(\sE_1)=0$, the upper semi-continuity theorem implies that $H^0(\sE_t)=0$ for all but finitely many $t\in \bA^1$. Then
we have $H^0(\sL_1)=0$, which contradicts to $\sL_1=\sO_X$. This proves that $\sE_t$, which is a deformation of $\sO_X\oplus \sI_C$, is a stable reflexive sheaf for generic $t$.
\end{proof}

Here some remarks are in order.

(1) The curve $C\subset X$ must have at least two connected components in order to fulfill the assumptions in the theorem.

(2) The ampleness of $D$ is not essential. As in the proof of Proposition \ref{log}, we only need $H^1(\sO_X(D))=H^2(\sO_X(D))=0$, which is equivalent to $H^1(\sO_D)=0$ since $X$ is a Calabi-Yau threefold.

(3) From diagram (\ref{comm11}), we know that $\sF(-D)$ is an ideal sheaf of a closed subscheme $W\subset X$. Therefore, from the bottom row of (\ref{comm20}), the constructed stable reflexive sheaf $\sE$ fits into
\[ 0\map\sO_X(-D)\map\sE\map\sI_W(D)\map 0\]
which is exactly the Serre construction.

\vskip 5pt

The following is a straightforward application of Theorem \ref{main}.

\begin{coro}\label{k3}
Let $\pi:X\map\bP^1$ be a K3-fibered Calabi-Yau threefold. For $k>1$, let $C_1,\cdots,C_k$ be Cohen-Macaulay curves with arithmetic genus $p_a(C_i)\ge 1$ in distinct fibers, generically locally complete intersection. Then there exists a rank $2$ stable reflexive sheaf $\sE$ on $X$ with $c_1(\sE)=0$ and $c_2(\sE)=[\cup C_i]$.
\end{coro}

To show this, we simply apply Theorem \ref{main} to the curve $C=\cup C_i$, and take $D$ be a fiber of $\pi$ disjoint from the curves $C_i$. Here $D$ is not ample. However, since $D$ is a K3 surface, $H^1(\sO_D)=0$. Hence item (2) in the previous remark applies. Suppose every curves $C_i$ has arithmetic genus $1$, then the obtained sheaf is locally free.

Here is another example. Let $\pi:X\to B$ be an elliptic fibered Calabi-Yau threefold. Let $f$ be the fiber class. Then there exists rank $2$ stable bundle $\sE$ on $X$ with $c_1(\sE)=0$, $c_2(\sE)=nf$ with $n\ge 2$. In fact, this is exactly the pull-back of a stable bundle over $B$ with $c_1=0$ and $c_2=n$; one can also verify directly this pull-back bundle is stable. Suppose $\alpha$ is a horizontal curve in $X$. Then one can construct stable reflexive sheaves on $X$ with $c_1=0$ and $c_2=\alpha+nf$ under some mild conditions. However, the construction of a curve $C$ in the class $\alpha+nf$ fulfill the requirement in Theorem \ref{main} is technical and we shall not pursuit in this direction.

\vskip 5pt

In the remainder of this section, we shall prove the main result of this paper which is an application of Theorem \ref{main}.

\begin{thm}\label{result}
Let $D\subset X$ be a very ample divisor on a Calabi-Yau threefold $X$. Then there exists a rank $r\ge 2$ stable reflexive sheaf $\sE$ on $X$ with $c_1(\sE)=0$, $c_2(\sE)=n[D^2]$ and $c_3(\sE)=2nD^3$ for any $n\ge r$.
\end{thm}

We shall prove it by induction on the rank $r$. First we consider rank $2$ case.

\begin{prop}\label{ini}
Let $D\subset X$ be a very ample divisor on a Calabi-Yau threefold $X$. Then there exists a rank $2$ stable reflexive sheaf $\sE$ on $X$ with $c_1(\sE)=0$, $c_2(\sE)=n[D^2]$ and $c_3(\sE)=2nD^3$ for any $n\ge 2$.
\end{prop}

\begin{proof}
We shall apply Theorem \ref{main} to construct this sheaf. It suffices to find a curve in the class $n[D^2]$ so that it satisfies the assumptions in Theorem \ref{main}.

%Since $D$ is very ample, we can find another divisor $D_1$ so that they span a $\bP^1$ in the linear system $|D|$. Let $C_0=D\cap D_1$. Let $D_2,\cdots, D_n$ be divisors representing different points in this $\bP^1$, distinct from that of $D$ and $D_1$. Then $C_0\subset D_i$ for all $i$.

%Again since $D$ is very ample, the divisors $C_0\subset D_i$ are also very ample. Hence we can deform $C_0$ to irreducible curves $C_i\subset D_i$ so that $C_i\cap D$ are zero dimensional. Moreover, we can arrange these curve $C_i$ so that they are pairwise disjoint. Now let $C=\cup_{i=1}^n C_i$. Since $n\ge 2$, $H^0(\sI_C(D))=0$. Moreover, the scheme-theoretic intersection $C\cap D$ is contained in $C_0$, the zero locus of a section of $N_{D/X}$. Using Riemann-Roch, we know that $\deg\omega_{C_i}(-D)=D^3>0$. Now we can apply Theorem \ref{main} to the curve $C$ and the ample divisor $D$ to get a stable reflexive sheaf $\sE$ from a deformation of $\sO_X\oplus \sI_C$. By a straightforward calculation, $c_3(\sE)=2p_a(C)-2=2nD^3$.

Since $D$ is very ample, we can assume that it is effective and smooth by Bertini's theorem. We can also find another smooth divisor $S_1\subset X$ in $|D|$ so that they span a $\bP^1$ in the linear system $|D|$. Let $C_0=D\cap S_1$. Let $S_2,\cdots, S_n$ be divisors representing different points in this $\bP^1$, distinct from that of $D$ and $S_1$. Then $C_0\subset S_i$ for all $i$.

Again since $D$ is very ample, the divisors $C_0\subset S_i$ are also very ample. Hence we can deform $C_0$ to irreducible curves $C_i\subset S_i$ so that $C_i\cap D$ are zero dimensional. Moreover, we can arrange these curve $C_i$ so that they are pairwise disjoint. Now let $C=\cup_{i=1}^n C_i$. Since $n\ge 2$, $H^0(\sI_C(D))=0$. Moreover, the scheme-theoretic intersection $C\cap D$ is contained in $C_0$, which is the zero locus of a section of $N_{D/X}$. Using Riemann-Roch, we know that $\deg\omega_{C_i}(-D)=D^3>0$. Hence there is a section of $\omega_{C_i}(-D)$ which generates $\omega_{C_i}(-D)$ generically. Now we can apply Theorem \ref{main} to the curve $C$ and the ample divisor $D$ to get a stable reflexive sheaf $\sE$ from a deformation of $\sO_X\oplus \sI_C$. By a straightforward calculation, $c_3(\sE)=2p_a(C)-2=2nD^3$.
\end{proof}

We shall use the following generalization of Theorem \ref{main} to do the induction.

\begin{prop}\label{ind}
Let $\sF$ be a stable sheaf on a Calabi-Yau threefold $X$ with $c_1(\sF)=0$ and not reflexive. Let $D\subset X$ be an ample divisor. Suppose it satisfies the following conditions:

(1) there exists an extension $0\map\sF\map\widetilde\sF\map\sO_D\map 0$ with $H^0(\widetilde\sF)=0$, and

(2) the natural homomorphism $\Ext^1(\sF,\sO_X(-D))\to\Ext^1(\sF,\sO_X)$ is nonzero, and its image contains an element corresponding to a reflexive sheaf.

Then $\sO_X\oplus \sF$ deforms to a stable reflexive sheaf $\sE$ fitting into
\[ 0\map \sO_X(-D)\map\sE\map\widetilde{\sF}\map 0.\]
\end{prop}

The proof follows from that of Theorem \ref{main}, with the ideal sheaf $\sI_C $ replaced by $\sF$ here. We shall not repeat it.

\vskip 5pt

Now we give the proof of Theorem \ref{result}.

\begin{proof}[Proof  of Theorem \ref{result}]
We prove it by induction on the rank $r$. For $r=2$, the result follows from Proposition \ref{ini}.

Now suppose for an $r\ge 2$ and every $n\ge r$, we have constructed rank $r$ stable reflexive sheaves $\sE_r$ from a deformation of $\sO_X\oplus \sF_{r-1}$, so that $c_1(\sE_r)=0$ and $c_2(\sE_r)=n[D^2]$. Furthermore, there exists an ample divisor $D_{r-1}$ in the linear system $|D|$, so that $\sE_r$ fits into the following diagram
\begin{equation}
\begin{CD}\label{comm4}
0@>>>\sO_X(-{D_{r-1}})@>>>G@>>>\sF_{r-1}@>>>0\\
@. @| @VVV @VVV \\
0@>>>\sO_X(-{D_{r-1}})@>>>\sE_r@>>>\widetilde{\sF_{r-1}}@>>>0\\
@.@.@VVV @VVV\\
@.@.\sO_{D_{r-1}}@=\sO_{D_{r-1}}
\end{CD}
\end{equation}
with $H^0(\widetilde{\sF_{r-1}})=0$.% and $\Ext^1(\sF_{r-1},\sO_X(-D))\to\Ext^1(\sF_{r-1},\sO_X)$ is nonzero.

In order to apply Proposition \ref{ind} to find a rank $r+1$ stable reflexive sheaf, we need to construct a sheaf $\sF_r$ which fulfill the assumptions there. For this, we make the following induction assumptions:

(A1) We can find other divisors $S_1,\cdots, S_n$ in $|D|$, distinct from $D_{r-1}$, so that they span a pencil that contains the divisor $D_{r-1}$. Let $C_0$ be the intersection of these divisors. The singular locus of $\sF_{r-1}$ is contained in the union $C=\cup_{i=1}^n C_i$ with $C_i\subset S_i$, which has the configuration type described in the proof of Proposition \ref{ini}. The singular locus of $\widetilde{\sF_{r-1}}$ is contained in $C_0\cup C$.

(A2) For every other divisor $S'$ in the pencil, and a curve $C'\subset S'$ which is disjoint from $C$, we can find a surjective homomorphism $\sF_{r-1}\to \sO_{C'}$.

To get a picture of these assumptions, it is helpful to look at the case $r=2$. Obviously, $\sF_{r-1}=\sF_1=\sI_C$ satisfies these assumptions (c.f. proof of Proposotion \ref{ini}), and $\widetilde{\sF_1}=\sI_W(D_1)$ with $W=C\cup C_0$ (c.f. Remark (3) after Theorem \ref{main}).

\vskip 5pt

Now we construct a rank $r$ stable sheaf $\sF_r$ with $c_2(\sF_r)=(n+1)[D^2]$ and satisfies the assumptions (A1) and (A2). We shall show that $\sF_r$ satisfies the conditions in Proposition \ref{ind} as well. Therefore $\sO_X\oplus \sF_r$ deforms to a rank $r+1$ stable reflexive sheaf $\sE_{r+1}$.

Based on (A2), we can find a curve $C'\subset S'$, linearly equivalent to $C_0$ in $S'$, and a surjective homomorphism $\sF_{r-1}\to \sO_{C'}$. Accordingly, from diagram (\ref{comm4}), we can find a surjective homomorphism $\widetilde{\sF_{r-1}}\to \sO_{C'}$. For $r=2$, this is given by $\sI_W(D_1)\to \sO_{C'}$, which is obvious since the restriction $\sI_W(D_1)|_{C'}$ is the direct sum of $\sO_{C'}$ and torsion part. It induces a surjective homomorphism $\sE_r\to \sO_{C'}$. Define
\[ \sF_r=\ker(\sE_r\map \sO_{C'}). \]
Evidently, $c_1(\sF_r)=0$, $c_2(\sF_r)=c_2(\sE_r)-c_2(\sO_{C'})=(n+1)[D^2]$, and $\sF_r$ is not reflexive. Moreover, $\sF_r$ is stable because $\sE_r$ is stable. It is clear that $\sF_r$ satisfies the induction assumptions (A1) and (A2).

Accordingly, we denote $\sK_{r-1}$ (resp. $\widetilde{\sK_{r-1}}$, $G'$) be the kernel of ${\sF_{r-1}}\to \sO_{C'}$ (resp. $\widetilde{\sF_{r-1}}\to \sO_{C'}$, $G\to\sO_{C'}$). Then we obtain the following
\begin{equation}
\begin{CD}\label{comm5}
0@>>>\sO_X(-{D_{r-1}})@>>>G'@>>>\sK_{r-1}@>>>0\\
@. @| @VVV @VVV \\
0@>>>\sO_X(-{D_{r-1}})@>>>\sF_r@>>>\widetilde{\sK_{r-1}}@>>>0\\
@.@.@VVV @VVV\\
@.@.\sO_{D_{r-1}}@=\sO_{D_{r-1}}
\end{CD}
\end{equation}

Now we verify condition (1) in Proposition \ref{ind}. Notice that we can find (infinitely many) divisor $D'$, linearly equivalent to $D$, so that there is an exact sequence
\[0\map \sK_{r-1}\map{(\sK_{r-1})'}\map\sO_{D'}\map 0\]
with $H^0((\sK_{r-1})')=0$. By pushing forward this sequence via the injective homomorphism $\sK_{r-1}\to \widetilde{\sK_{r-1}}$, we obtain a sheaf $\widetilde{\sF'}$ fitting into
\begin{equation}
\begin{CD}\label{comm6}
0@>>>\sK_{r-1}@>>>{(\sK_{r-1})'}@>>>\sO_{D'}@>>>0\\
@. @VVV @VVV @| \\
0@>>>\widetilde{\sK_{r-1}}@>>>\widetilde{\sF'}@>>>\sO_{D'}@>>>0\\
@.@VVV @VVV\\
@.\sO_{D_{r-1}}@=\sO_{D_{r-1}}
\end{CD}
\end{equation}
Let $u:H^0(\sO_{D_{r-1}})\to H^1(\sK_{r-1})$ be the induce homomorphism from the left column. Choose $D'$ so that the composition of $u$ with $H^1(\sK_{r-1})\to H^1((\sK_{r-1})')$ is nonzero, hence injective. It implies that by the middle column, we have $H^0(\widetilde{\sF'})=0$.

Then consider the middle row in (\ref{comm5}),
\begin{equation}
0\map\sO_X(-{D_{r-1}})\map\sF_r\map\widetilde{\sK_{r-1}}\map 0.
\end{equation}
Applying $\Hom(\sO_{D'},-)$, since $D$ is ample, $\Ext^2(\sO_{D'},\sO_X(-{D_{r-1}}))=0$ by Kodaira vanishing theorem. We obtain a surjective homomorphism
\begin{equation}\label{surj}
\Ext^1(\sO_{D'},\sF_r)\map \Ext^1(\sO_{D'},\widetilde{\sK_{r-1}}).
\end{equation}
Notice that the central row in (\ref{comm6}) is an element in $\Ext^1(\sO_{D'},\widetilde{\sK_{r-1}})$, by the surjectivity of (\ref{surj}), we can find $\widetilde{\sF_{r}}$ fitting into
\begin{equation}
\begin{CD}\label{comm7}
@.\sO_X(-{D_{r-1}})@=\sO_X(-{D_{r-1}})\\
@.@VVV @VVV\\
0@>>>\sF_{r}@>>>\widetilde{\sF_{r}}@>>>\sO_{D'}@>>>0\\
@. @VVV @VVV @| \\
0@>>>\widetilde{\sK_{r-1}}@>>>\widetilde{\sF'}@>>>\sO_{D'}@>>>0\\
\end{CD}
\end{equation}
From the middle column, we get $H^0(\widetilde{\sF_{r}})=0$. Hence the middle row in (\ref{comm7}) is the required extension which satisfies condition (1) in Proposition \ref{ind}. We denote $D_r=D'$.

\vskip 5pt

For condition (2) in Proposition \ref{ind}, we consider the sequence
\[0\map\sF_r\map \sE_r\map\sO_{C'}\map 0. \]
Applying $\Hom(-,\sO_X(-D_r))$ and $\Hom(-,\sO_X)$, we get the following
\begin{equation}
\begin{CD}\label{comm38}
\Ext^1(\sE_r,\sO_X(-D_r))@>\phi>>\Ext^1(\sF_{r},\sO_X(-D_r))@>\rho>>\Ext^2(\sO_{C'},\sO_X(-D_r))\\
 @VVV @VVf_{r}V @VVfV \\
\Ext^1(\sE_r,\sO_X)@>\phi'>>\Ext^1(\sF_{r},\sO_X)@>>>\Ext^2(\sO_{C'},\sO_X)\\
\end{CD}
\end{equation}
where $\phi$ and $\phi'$ are injective maps.

First we claim that every nonreflexive extension in $\Ext^1(\sF_{r},\sO_X)$ lies in the image of $\phi'$. Indeed, for any extension
\[0\map\sO_X\map \sH\map\sF_r\map 0 \]
so that $\sH$ is not reflexive. Taking double dual, we have
\begin{equation}
\begin{CD}\label{comm21}
0@>>>\sO_X@>>> \sH@>>>\sF_r@>>> 0 \\
@. @| @VVV @VVV  \\
0@>>>\sO_X@>>> (\sH^{\!\vee})^{\!\vee}@>h>>\sE_r  \\
@. @. @VVV @VVV  \\
@. @. *@>>>\sO_{C'}  \\
\end{CD}
\end{equation}
We want to show that $h$ is surjective. If not, let $(\sE_r)'$ be its image. Then $(\sE_r)'\neq \sF_r$ since $\sH$ is not reflexive. On the other hand, by the inclusion $\sF_r\subset (\sE_r)'\subset\sE_r$, and $\sE_r/\sF_r=\sO_{C'}$, we know that $\sQ=\sE_r/(\sE_r)'$ is a quotient of $\sO_{C'}$. Then we get
\[ 0\map (\sE_r)'\map\sE_r\map \sQ\map 0 \]
where $\sQ$ is a nonzero sheaf support at points. Applying $\Hom(-,\sO_X)$, we get
\[\Ext^1(\sE_r,\sO_X)\map\Ext^1((\sE_r)',\sO_X)\map\Ext^2(\sQ,\sO_X)=0. \]
The surjectivity of the first map implies that we can find a sheaf $\sK$ fitting into
\begin{equation}
\begin{CD}\label{comm21}
0@>>>\sO_X@>>> (\sH^{\!\vee})^{\!\vee}@>>>(\sE_r)' @>>>0 \\
@. @| @VVV @VVV  \\
0@>>>\sO_X@>>> \sK@>>>\sE_r @>>>0 \\
@. @. @VVV @VVV  \\
@. @. \sQ@>>>\sQ  \\
\end{CD}
\end{equation}
This is not possible. Since $(\sH^{\!\vee})^{\!\vee})$ is a reflexive sheaf on a smooth threefold, it admits a two term locally free resolution. Therefore $\Ext^1(\sQ, (\sH^{\!\vee})^{\!\vee})=0$, it implies that $\sK\cong\sQ\oplus(\sH^{\!\vee})^{\!\vee}$, which contradicts to the middle row. Therefore $h$ is surjective, and every nonreflexive extension in $\Ext^1(\sF_{r},\sO_X)$ lies in the image of $\phi'$ which proves the claim.

Next we show that $\phi$ in (\ref{comm38}) is not surjective.

Using diagram (\ref{comm4}), we obtain
\begin{equation}\label{dim1}
\dim\Ext^1(\sE_r,\sO_X(-{D_r}))=\dim\Ext^1(\sF_{r-1},\sO_X(-{D_r}));
\end{equation}
and using (\ref{comm5}), we get
\begin{equation}\label{dim2}
\dim\Ext^1(\sF_r,\sO_X(-{D_r}))=\dim\Ext^1({\sK_{r-1}},\sO_X(-{D_r}));
\end{equation}
By the sequence
\[ 0\map{\sK_{r-1}}\map\sF_{r-1}\map\sO_{C'}\map 0 \]
we have
\[ 0\map\Ext^1(\sF_{r-1},\sO_X(-{D_r}))\to \Ext^1({\sK_{r-1}},\sO_X(-{D_r}))\map * \]
By induction, we know that
\[\dim \Ext^1(\sF_{r-1},\sO_X(-{D_r}))<\dim \Ext^1({\sK_{r-1}},\sO_X(-{D_r})).\]
Combined with (\ref{dim1}) and (\ref{dim2}), we get
\[\dim\Ext^1(\sE_r,\sO_X(-{D_r}))< \dim\Ext^1(\sF_r,\sO_X(-{D_r})).\]
Hence $\phi$ is not surjective.

To verify condition (2) in Proposition \ref{ind}, note that $C'$ does not contained in $D_r$, it implies that $f$ in (\ref{comm38}) is injective. Therefore, we can find $\theta\in \Ext^1(\sF_r,\sO_X(-D_r))$ which is not in the image of $\phi$. Since $f$ is injective, $f(\rho(\theta))$ is nonzero. It follows that $f_r(\theta)$ is nonzero as well. Since $f_r(\theta)$ does not contained in the image of $\phi'$, we know that it corresponds to a reflexive sheaf.

Finally, applying Proposition \ref{ind} to the nonreflexive sheaf $\sF_r$, we obtain a stable reflexive sheaf of rank $r+1$ from the deformation of $\sO_X\oplus\sF_r$.
\end{proof}

\section{Further discussions}

For Calabi-Yau threefolds of Picard number $1$, it is relatively easier to verify directly the stability of a bundle. However, for general case, there is no standard procedure to check stability. This is the reason the deformation argument plays a key role.

As we already mentioned, the main Theorem of this paper is an application of the technical result Theorem \ref{main}. For special Calabi-Yau threefolds, e.g. elliptic fibered or branched covering of Fano threefolds, one may use Theorem \ref{main} to construct stable bundles or reflexive sheaves with other given Chern classes, as we briefly did in Corollary \ref{k3} and the discussion afterwards. It is also interesting to compare the result with the spectral cover construction for elliptic Calabi-Yau \cite{FMW99,GHY14,AC12}, and results obtained by elementary transformation \cite{Nak04}.

For rank $r\ge 3$, Maeda \cite[Theorem 2.5]{Mae90} showed that one can construct vector bundles of rank $r$ using a generalized Serre construction. We expect one can combine his construction with Theorem \ref{main} in this note to get some refined results on the existence of stable bundles instead of reflexive sheaves.

In Theorem \ref{result}, applying deformation of the representative curve in the class $n[D^2]$, one may get curves of lower arithmetic genus and therefore obtain reflexive sheaves with lower values of $c_3$. The existence of such deformation depends on the configuration of the curve and the Calabi-Yau threefold. We hope to come back to these questions in future research.

\bibliographystyle{alpha}

\end{document}